\documentclass{article}%

\usepackage{amsmath}%
\usepackage{amsfonts}%
\usepackage{amssymb}%

\usepackage{amsthm}

\usepackage{graphicx}
\usepackage{todonotes}
\usepackage{a4wide}
\usepackage{url}
\usepackage[utf8]{inputenc}
\usepackage{hyperref}
\newcommand{\ElThree}{L_0}

\newtheorem{theorem}{Theorem}
\numberwithin{theorem}{section}

\newtheorem{assumption}[theorem]{Assumption}

\newtheorem{corollary}[theorem]{Corollary}

\newtheorem{definition}[theorem]{Definition}

\newtheorem{lemma}[theorem]{Lemma}

\newtheorem{proposition}[theorem]{Proposition}

\usepackage[affil-it]{authblk}

\theoremstyle{definition}
\newtheorem{remark}[theorem]{Remark}

\newcommand{\abGal}[1] {\operatorname{Gal}\big(\overline{#1}/#1\big)}

\begin{document}

\title{An explicit open image theorem for products of elliptic curves}

\author{Davide Lombardo
  \thanks{\texttt{davide.lombardo@math.u-psud.fr}}}
\affil{Département de Mathématiques d'Orsay}

\date{}
\maketitle

\begin{abstract}
Let $K$ be a number field and $E_1, \ldots, E_n$ be elliptic curves over $K$, pairwise non-isogenous over $\overline{K}$ and without complex multiplication over $\overline{K}$.
We study the image $G_\infty$ of the adelic representation of $\abGal{K}$ naturally attached to $E_1 \times \cdots \times E_n$. The main result is an explicit bound for the index of $G_\infty$ in $\left\{ (x_1,\ldots,x_n) \in  \operatorname{GL}_2(\hat{\mathbb{Z}})^n \bigm\vert \det x_i = \det x_j \;\; \forall i,j \right\}$.
\end{abstract}

\section{Introduction}

In this work we prove an explicit, adelic surjectivity result for the Galois representation attached to a product of pairwise non-isogenous, non-CM elliptic curves, extending the result of \cite{AdelicEC}. Our main theorem is as follows:
\begin{theorem}\label{thm_Explicit}
Let $E_1, \ldots, E_n$, $n \geq 2$, be elliptic curves defined over a number field $K$, pairwise not isogenous over $\overline{K}$. Suppose that $\operatorname{End}_{\overline{K}}(E_i)=\mathbb{Z}$ for $i=1,\ldots,n$, and denote by $G_\infty$ the image of $\abGal{K}$ inside
\[\prod_{i=1}^n \prod_{\ell}  \operatorname{Aut}(T_\ell(E_i)) \subset \operatorname{GL}_2(\hat{\mathbb{Z}})^n.\]
Set $\gamma:=10^{13}$, $\delta:=\exp \exp \exp (12)$, and let
$H=\max\left\{ 1, \log[K:\mathbb{Q}], \max_{i} h(E_i) \right\}$, where $h(E_i)$ denotes the stable Faltings height of $E_i$. The group $G_\infty$ has index at most
\[
\delta^{n(n-1)} \cdot \left([K:\mathbb{Q}] \cdot H^2 \right)^{\gamma n(n-1)} 
\]
in
\[
\Delta:=\left\{ (x_1,\ldots,x_n) \in  \operatorname{GL}_2(\hat{\mathbb{Z}})^n \bigm\vert \det x_i = \det x_j \;\; \forall i,j \right\}.
\]
\end{theorem}

\begin{remark}
Note that the compatibility of the Weil pairing with the action of Galois forces $G_\infty$ to be contained in $\Delta$. Also note that we shall prove slightly more precise statements (see lemma \ref{lemma_Estimatesb0} and theorem \ref{thm_FinaleNCurves} below), which immediately imply theorem \ref{thm_Explicit} by proposition \ref{prop_b0} and elementary estimates.

\end{remark}

It should be noted that it has been known since the work of Serre and Masser-W\"ustholz (cf. \cite{MR1209248}, Main Theorem and Proposition 1) that the isogeny theorem (section \ref{sect_IsogenyThms} below) gives an effective bound $\ell_0$ on the largest prime $\ell$ for which the image of the representation
\[
\abGal{K} \to \operatorname{Aut}(T_\ell(E_1 \times \cdots \times E_n))
\]
does not contain $\operatorname{SL}_2(\mathbb{Z}_\ell)^n$. As it was in \cite{AdelicEC}, the main difficulty in proving theorem \ref{thm_Explicit} lies in controlling the image of the representation modulo powers of primes smaller than $\ell_0$.

The proof of theorem \ref{thm_Explicit} is somewhat technical, so before fiddling with the details we describe the main ideas behind it. The general framework is the same as that of the proof of the non-effective open image theorem for such a product (cf. for example \cite[Theorem 3.5]{MR0419358}), with the added difficulties that naturally arise when trying to actually compute the index. In particular, when writing `of finite index' or `open' in the sketch that follows we tacitly imply that the index in question is explicitly computable in terms of the data. Whenever the need arises to actually quantify indices, it will be useful to work with the following `standard' open subgroups:

\begin{definition}
For a prime $\ell$ and a positive integer $s$ we let $\mathcal{B}_\ell(s)$ be the open subgroup of $\operatorname{SL}_2(\mathbb{Z}_\ell)$ given by
\[
\left\{ x \in \operatorname{SL}_2(\mathbb{Z}_\ell) \bigm\vert x \equiv \operatorname{Id} \pmod{\ell^s} \right\}.
\]
We also set $\mathcal{B}_\ell(0)=\operatorname{SL}_2(\mathbb{Z}_\ell)$, and for non-negative integers $k_1, \ldots, k_n$ we denote by $\mathcal{B}_\ell(k_1,\ldots, k_n)$ the open subgroup $\prod_{j=1}^n \mathcal{B}_\ell(k_i)$ of $\operatorname{SL}_2(\mathbb{Z}_\ell)^n$.
\end{definition}

\smallskip

Let us now describe the proof method proper. 
It is not hard to see that it is enough to study the intersection $G_\infty \cap \operatorname{SL}_2(\mathbb{Z}_\ell)^n$, because the determinant of elements in $G_\infty$ can be easily understood in terms of the cyclotomic character. A short argument then shows that it suffices to consider products $E_1 \times E_2$ involving only two factors: this is done by proving that a subgroup of $\operatorname{SL}_2(\mathbb{Z}_\ell)^n$ whose projection on any \textit{pair} of factors is of finite index is itself of finite (and explicitly bounded) index. This step will be carried out in section \ref{sect_ReductionToTwo} below, and should be thought of as the `integral' version of \cite[Lemma on p. 790]{MR0457455}.

With this result at hand we are thus reduced to dealing with subgroups $G$ of $\operatorname{SL}_2(\mathbb{Z}_\ell) \times \operatorname{SL}_2(\mathbb{Z}_\ell)$ whose projections on either factor are of finite index in $\operatorname{SL}_2(\mathbb{Z}_\ell)$. Note that the fact that this index is finite is the open image theorem for a single elliptic curve, which was proved by Serre in \cite{MR0387283} and made explicit in \cite{AdelicEC}. We wish to show that $G$ is of (explicitly bounded) finite index in $\operatorname{SL}_2(\mathbb{Z}_\ell)^2$, that is, we want to produce a $t$ such that $G$ contains $\mathcal{B}_\ell(t,t)$: this clearly comes down to proving that the two kernels $K_i=\ker \left(G \xrightarrow{\pi_i}\operatorname{SL}_2(\mathbb{Z}_\ell) \right)$, when identified with subgroups of $\operatorname{SL}_2(\mathbb{Z}_\ell)$, are of (explicitly bounded) finite index. By symmetry, we just need to deal with $K_1$.


In section \ref{sect_PinkResults} we linearize the problem by reducing it to the study of certain $\mathbb{Z}_\ell$-Lie algebras: we give the statements of two technical results whose proof, being rather lengthy, is deferred to the companion paper \cite{GeneralPink}; while the results themselves are more complicated, the methods used to show them do not differ much from those of \cite{AdelicEC}.

A simple lemma, again given in section \ref{sect_PinkResults}, further reduces the problem of finding an integer $t$ such that $\mathcal{B}_\ell(t)$ is contained in $K_1$ to the (easier) question of finding a $t$ such that $K_1(\ell^t)$, the reduction modulo $\ell^t$ of $K_1$, is nontrivial. We exploit here the fact that $\pi_2(G)$ (the projection of $G$ on the second factor $\operatorname{SL}_2(\mathbb{Z}_\ell)$) acts by conjugation on $K_1$, the latter being a normal subgroup of $G$: we prove that a group whose reduction modulo $\ell^t$ is nontrivial and that is stable under conjugation by a finite-index subgroup of $\operatorname{SL}_2(\mathbb{Z}_\ell)$ must itself be of finite index in $\operatorname{SL}_2(\mathbb{Z}_\ell)$. This reduction step is made simpler by the fact that we can work with Lie algebras instead of treating the corresponding groups directly (which might be quite complicated).



Next we ask what happens if we suppose that the smallest integer $t$ such that $K_1(\ell^t)$ is nontrivial is in fact very large. The conclusion is that the Lie algebra of $G$ looks `very much like' the graph of a Lie algebra morphism $\mathfrak{sl}_2(\mathbb{Z}_\ell) \to \mathfrak{sl}_2(\mathbb{Z}_\ell)$, namely it induces an actual Lie algebra morphism when regarded modulo $\ell^N$ for a very large $N$ (depending on $t$). Following for example the approach of Ribet (cf.~the theorems on p.~795 of \cite{MR0457455}), we would like to know that all such morphisms are `inner', that is, they are given by conjugation by a certain matrix: it turns out that this is also true in our context, even though the result is a little less straightforward to state (cf.~section \ref{sect_AutomorphismsAreInner}).



In section \ref{sect_ProductOfTwo} we then deal with the case of two elliptic curves, applying the aforementioned results to deduce an open image theorem for each prime $\ell$. It is then an easy matter to deduce, as we do in section \ref{sect_Conclusion}, the desired adelic result for any finite product.


\medskip

\noindent\textbf{Notation.} Throughout the whole paper, the prime $2$ plays a rather special role, and special care is needed to treat it. In order to give uniform statements that hold for every prime we put $v=0$ or $1$ according to whether the prime $\ell$ we are working with is odd or equals 2, that is we set
\[
v=v_\ell(2)=\begin{cases} 0, \text{ if } \ell \text{ is odd} \\ 1, \text{ otherwise}. \end{cases}
\]
We will also consistently use the following notations: 
\begin{itemize}
\item $G_\ell$, to denote the image of $\abGal{K}$ in $\operatorname{Aut}T_\ell(E_1) \times \cdots \times \operatorname{Aut} T_\ell(E_k)$;
\item $G(\ell^n)$, where $G$ is a closed subgroup of a certain $\operatorname{GL}_2(\mathbb{Z}_\ell)^k$, to denote the reduction of $G$ modulo $\ell^n$, that is to say its image in $\operatorname{GL}_2(\mathbb{Z}/\ell^n\mathbb{Z})^k$;
\item $G'$, to denote the topological closure of the commutator subgroup of $G$;
\item $[g]$, where $g \in G$, to denote the reduction of $g$ modulo $\ell$, that is, its image in $\operatorname{GL}_2(\mathbb{F}_\ell)^k$.
\end{itemize}

\noindent\textbf{Acknowledgments.} The author gratefully acknowledges financial support from the Fondation Mathématique Jacques Hadamard (grant ANR-10-CAMP-0151-02 in the “Programme des Investissements d’Avenir”).

\section{Preliminaries on isogeny bounds}\label{sect_IsogenyThms}
The main tool that makes all the effective estimates possible is the isogeny theorem of Masser and W\"ustholz \cite{MR1217345} \cite{MR1207211}, which we employ in the explicit version proved in \cite{PolarisationsEtIsogenies}. We need some notation: we let $\alpha(g)=2^{10}g^3$ and define, for any abelian variety $A/K$ of dimension $g$,
\[
b(A/K)=b([K:\mathbb{Q}],g,h(A))=\left( (14g)^{64g^2} [K:\mathbb{Q}] \max\left(h(A), \log [K:\mathbb{Q}],1 \right)^2 \right)^{\alpha(g)}.
\]

\begin{theorem}{(\cite[Théorème 1.4]{PolarisationsEtIsogenies})}\label{thm_Isogeny}
Let $K$ be a number field and $A, A^*$ be two Abelian $K$-varieties of dimension $g$. If $A, A^*$ are isogenous over $K$, then there exists a $K$-isogeny $A^* \to A$ whose degree is bounded by $b([K:\mathbb{Q}],\dim(A),h(A))$.\end{theorem}

\begin{remark} As the notation suggests, the three arguments of $b$ will always be the degree of a number field $K$, the dimension $g$ of an Abelian variety $A/K$ and its stable Faltings height $h(A)$.
\end{remark}

We shall need a slight refinement of this bound. Following Masser \cite{MR1619802}, we introduce the following definition:
\begin{definition}
Let $A/K$ be an abelian variety. We say that $A$ is a TM-product over $K$ if $A$ is isomorphic (over $K$) to $A_1^{e_1} \times \cdots \times A_n^{e_n}$, where $A_1,\ldots,A_n$ are $K$-simple abelian varieties, mutually non-isogenous (over $K$) and with trivial endomorphism ring (over $K$).
\end{definition}

Adapting arguments given by Masser in \cite{MR1619802}, it is easy to prove
\begin{theorem}{(\cite[Theorem 2.4]{AdelicEC})}
Suppose that $A/K$ is a TM-product over $K$. Let $b \in \mathbb{R}$ be a constant with the following property: for every $K$-abelian variety $A^*$ isogenous to $A$ over $K$ there exists an isogeny $\psi:A^* \to A$ with $\deg \psi \leq b$. Then there exists an integer $b_0 \leq b$ with the following property: for every $K$-abelian variety $A^*$ isogenous to $A$ over $K$ there exists an isogeny $\psi_0:A^* \to A$ with $\deg \psi_0 \bigm\vert b_0$.
\end{theorem}

We will denote by $b_0(A/K)$ the minimal $b_0$ with the property of the above theorem; in particular $b_0(A/K) \leq b(A/K)$. 
Suppose now that, in addition to $A/K$ being a TM-product over $K$, its simple factors $A_i$ are \textit{absolutely} simple and pairwise non-isogenous \textit{over $\overline{K}$}. Then for any field extension $K'$ of $K$ the hypotheses of the previous theorem hold for $A_{K'}$, so it makes sense to consider the quantity $b_0(A/K')$ as $K'$ ranges through all the finite extensions of $K$ of degree bounded by $d$. Since $b_0(A/K') \leq b(d[K:\mathbb{Q}],h(A),\dim(A))$ stays bounded, the number
$\displaystyle \operatorname{lcm}_{[K':K] \leq d} b_0(A/K')$
is finite, and we give it a name:
\begin{definition}\label{def_b0} 
Let $A/K$ be an abelian variety such that $A$ is a TM-product over $K$, with simple factors that are absolutely simple and pairwise non-isogenous over $\overline{K}$. We set
\[b_0(A/K;d)=\operatorname{lcm}_{[K':K] \leq d} b_0(A/K').\]
\end{definition}

A slight modification of the arguments of \cite[Theorem D]{MR1619802}, combined with theorem \ref{thm_Isogeny}, gives
\begin{proposition}\label{prop_b0}{(\cite[Proposition 2.6]{AdelicEC})}
Let $A/K$ be a $g$-dimensional abelian variety that is isomorphic over $K$ to a product $A_1^{e_1} \times \cdots \times A_n^{e_n}$, where $A_1,\ldots,A_n$ are simple over $\overline{K}$, mutually non-isogenous over $\overline{K}$, and have trivial endomorphism ring over $\overline{K}$. Then we have
\[
b_0(A/K;d) \leq b(A/K;d):=4^{\exp(1) \cdot (d(1+\log d )^2)^{\alpha(g)} } b([K:\mathbb{Q}],\dim(A),h(A))^{1+ \alpha(g) \log(d(1+\log d)^2) }.
\]
\end{proposition}



\section{An integral Goursat-Ribet lemma for $\operatorname{SL}_2(\mathbb{Z}_\ell)$}\label{sect_ReductionToTwo}
As anticipated, a (necessary and) sufficient condition for a closed subgroup of $\operatorname{SL}_2(\mathbb{Z}_\ell)^n$ to be open is that all its projections on pairs of factors $\operatorname{SL}_2(\mathbb{Z}_\ell)^2$ are themselves open. This is precisely the content of the following lemma:

\begin{lemma}{(\cite[Lemma 2.9]{GeneralPink})}
Let $n$ be a positive integer, $G$ a closed subgroup of $\prod_{i=1}^n \operatorname{SL}_2(\mathbb{Z}_\ell)$, and $\pi_i$ the projection from $G$ on the $i$-th factor. Suppose that, for every $i\neq j$, the group $\left(\pi_i \times \pi_j\right)(G)$ contains $\mathcal{B}_\ell(s_{ij}, s_{ij})$ for a certain non-negative integer $s_{ij}$ (with $s_{ij} \geq 2$ if $\ell=2$ and $s_{ij} \geq 1$ if $\ell=3$): then $G$ contains $\prod_{i=1}^n \mathcal{B}_\ell\left(\sum_{j \neq i} s_{ij}+(n-2)v\right)$.
\end{lemma}

\begin{corollary}\label{cor_ProductsOfManyCurves}
Let $G$ be a closed subgroup of $\prod_{i=1}^n \operatorname{SL}_2(\hat{\mathbb{Z}})$ with $n \geq 2$. Suppose that for every pair of indices $i \neq j$ there exists a subgroup $S^{(i,j)}$ of $\operatorname{SL}_2(\hat{\mathbb{Z}})^2$ with the following properties:
\begin{itemize}
\item the projection of $G$ on the direct factor $\operatorname{SL}_2(\hat{\mathbb{Z}}) \times \operatorname{SL}_2(\hat{\mathbb{Z}})$ corresponding to the pair of indices $(i,j)$ contains $S^{(i,j)}$;
\item $S^{(i,j)}$ decomposes as a direct product $\prod_{\ell \text{ prime}} S^{(i,j)}_\ell \subseteq \prod_{\ell} \operatorname{SL}_2(\mathbb{Z}_\ell)^2$;
\item for every prime $\ell$ the group $S^{(i,j)}_\ell$ is of the form $\mathcal{B}_\ell(f^{(i,j)}_\ell,f^{(i,j)}_\ell)$, where $f_\ell^{(i,j)}$ is a non-negative integer, with $f^{(i,j)}_2 \geq 2$ if $\ell = 2$ and $f^{(i,j)}_3 \geq 1$ if $\ell=3$.
\item for almost every $\ell$, the group $S^{(i,j)}_\ell$ is all of $\operatorname{SL}_2(\mathbb{Z}_\ell) \times \operatorname{SL}_2(\mathbb{Z}_\ell)$ (so $f_\ell^{(i,j)}=0$);
\end{itemize}
Denote by $c^{(i,j)}$ the index of $S^{(i,j)}$ in $\operatorname{SL}_2(\hat{\mathbb{Z}}) \times \operatorname{SL}_2(\hat{\mathbb{Z}})$ and $\displaystyle c= \max_{i \neq j} c^{(i,j)}$. The index of $G$ in $\prod_{i=1}^n \operatorname{SL}_2(\hat{\mathbb{Z}})$ is strictly less than
$
2^{3n(n-2)}\zeta(2)^{n(n-1)} c^{n(n-1)/2}.
$
\end{corollary}

\begin{proof}
Let $\ell>3$ be a prime. If $S^{(i,j)}_\ell = \operatorname{SL}_2(\mathbb{Z}_\ell)^2$ for all $(i,j)$, then the previous lemma (with $s_{ij}=0$ for every pair of indices $(i,j)$) shows that $\prod_{k=1}^n \operatorname{SL}_2(\mathbb{Z}_\ell)$ is contained in $G$. 
Suppose on the other hand that either $\ell \leq 3$ or for at least one pair $(i,j)$ we have $S_\ell^{(i,j)} \neq \operatorname{SL}_2(\mathbb{Z}_\ell) \times \operatorname{SL}_2(\mathbb{Z}_\ell)$. 
The previous lemma tells us that the projection of $G$ on the direct factor $\prod_{i=1}^n \operatorname{SL}_2(\mathbb{Z}_\ell)$ of $\prod_{i=1}^n \operatorname{SL}_2(\hat{\mathbb{Z}})$ contains
\[
B_\ell\left(\sum_{j \neq 1} f_\ell^{(1,j)} + (n-2)v, \; \cdots,\; \sum_{j \neq n} f_\ell^{(n,j)} + (n-2)v\right)=\prod_{i} \mathcal{B}_\ell\left( \sum_{j \neq i} f_\ell^{(i,j)}+(n-2)v\right).
\]
Notice that the index of $\mathcal{B}_\ell(s)$ in $\operatorname{SL}_2(\mathbb{Z}_\ell)$, for $s \geq 1$, is $(\ell^2-1) \ell^{1+3(s-1)}<\ell^{3s}$, so the index of the above product in $\prod_{i=1}^n \operatorname{SL}_2(\mathbb{Z}_\ell)$ is bounded by
\[
\prod_{i=1}^n \left( \ell^{3 \sum_{j \neq i} f^{(i,j)}_\ell + 3(n-2)v}\right) = 2^{3n(n-2)v} \prod_{i=1}^n \prod_{j \neq i} \ell^{3 f_\ell^{(i,j)}}.
\]
Let now $\mathcal{P}=\left\{ 2,3 \right\} \cup \left\{ \ell \bigm\vert \exists (i,j) : S_\ell^{(i,j)} \neq \operatorname{SL}_2(\mathbb{Z}_\ell) \times \operatorname{SL}_2(\mathbb{Z}_\ell) \right\}$. By what we have just seen,
\[
\left[\prod_{k=1}^n \operatorname{SL}_2(\hat{\mathbb{Z}}) : G \right] \leq 2^{3n(n-2)} \prod_{\ell \in \mathcal{P}} \prod_{i=1}^n \prod_{j \neq i} \ell^{3 f_\ell^{(i,j)}}.
\]


On the other hand, note that the index of $S_\ell^{(i,j)}$ in $\operatorname{SL}_2(\mathbb{Z}_\ell) \times \operatorname{SL}_2(\mathbb{Z}_\ell)$ is at least $\ell^{6f_\ell^{(i,j)}} \cdot \left(\frac{\ell^2-1}{\ell^2}\right)^2$ (with equality if $f_\ell^{(i,j)}\geq 1$), so the above product is bounded by
\[
\begin{aligned}
2^{3n(n-2)}\prod_{\ell \in \mathcal{P}} \prod_{i<j} & \left\{ \left[ \operatorname{SL}_2(\mathbb{Z}_\ell)^2 : S_\ell^{(i,j)} \right] \cdot \left(\frac{\ell^2}{\ell^2-1}\right)^2 \right\} \\ & < 2^{3n(n-2)} \prod_{\ell} \left(\frac{\ell^2}{\ell^2-1}\right)^{n(n-1)} \cdot \prod_{i<j} \prod_{\ell \in \mathcal{P}} \left[ \operatorname{SL}_2(\mathbb{Z}_\ell)^2 : S_\ell^{(i,j)} \right] \\
& \leq 2^{3n(n-2)} \zeta(2)^{n(n-1)} \prod_{i < j} c^{(i,j)} \\
& \leq 2^{3n(n-2)} \zeta(2)^{n(n-1)} c^{n(n-1)/2}.
\end{aligned}
\]
\end{proof}


\section{Lie subalgebras of $\mathfrak{sl}_2(\mathbb{Z}_\ell)^n$ and some Pink-type results}\label{sect_PinkResults}
Let us briefly recall the construction (essentially due to Pink) of the $\mathbb{Z}_\ell$-Lie algebra associated with a subgroup of $\operatorname{GL}_2(\mathbb{Z}_\ell)^n$:
\begin{definition}{(cf.~\cite{MR1241950})}
Let $\ell$ be a prime. Define maps $\Theta_n$ as follows:
\[
\begin{array}{cccc}
\Theta_n : & \operatorname{GL}_2(\mathbb{Z}_\ell)^n  & \to & \bigoplus_{i=1}^n \mathfrak{sl}_2(\mathbb{Z}_\ell)\\
& (g_1,\ldots,g_n) & \mapsto & \left( g_1 -\frac{1}{2} \operatorname{tr}(g_1), \ldots, g_n -\frac{1}{2} \operatorname{tr}(g_n)  \right).
\end{array}
\]
If $G$ is a closed subgroup of $\operatorname{GL}_2(\mathbb{Z}_\ell)^n$ (resp.~of $\mathcal{B}_2(1,\ldots,1)$ in case $\ell=2$), define $L(G) \subseteq \mathfrak{sl}_2(\mathbb{Z}_\ell)^n$ to be the $\mathbb{Z}_\ell$-span of $\Theta_n(G)$. We call $L(G)$ the \textbf{Lie algebra} of $G$.
\end{definition}

The importance of this construction lies in the fact that it allows us to linearize the problem of showing that certain subgroups of $\operatorname{GL}_2(\mathbb{Z}_\ell)^n$ contain an explicit open neighbourhood of the identity: indeed, we have the following two results, for whose proof we refer the reader to \cite{GeneralPink}.

\begin{theorem}\label{thm_PinkGL22}{(\cite[Theorem 3.1]{GeneralPink})}
Let $\ell>2$ be a prime number and $G$ be a closed subgroup of $\operatorname{GL}_2(\mathbb{Z}_\ell) \times \operatorname{GL}_2(\mathbb{Z}_\ell)$. Let $G_1, G_2$ be the two projections of $G$ on the two factors $\operatorname{GL}_2(\mathbb{Z}_\ell)$, and let $n_1,n_2$ be positive integers such that $G_i$ contains $\mathcal{B}_\ell(n_i)$ for $i=1,2$. Suppose furthermore that for every $(g_1,g_2) \in G$ we have $\det(g_1)=\det(g_2)$. At least one of the following holds:
\begin{itemize}
\item $G$ contains $\mathcal{B}_\ell(20\max\{n_1,n_2\},20\max\{n_1,n_2\})$

\item there exists a subgroup $T$ of $G$, of index dividing $2 \cdot 48^2$, with the following properties:
\begin{itemize}
\item if $L(T)$ contains $\ell^k \mathfrak{sl}_2(\mathbb{Z}_\ell) \oplus \ell^k \mathfrak{sl}_2(\mathbb{Z}_\ell)$ for a certain integer $k$, then $T$ contains $\mathcal{B}_\ell(p,p)$, where
\[
p=2k+\max\left\{2k,8n_1,8n_2 \right\}.
\]We call this property $(\ast)$.
\item for any $(t_1,t_2)$ in $T$, if both $[t_1]$ and $[t_2]$ are multiples of the identity, then they are equal;
\item for any $(t_1,t_2)$ in $T$, the determinant of $t_1$ is a square in $\mathbb{Z}_\ell^\times$ (hence the same is true for the determinant of $t_2$).
\end{itemize}

\end{itemize}
\end{theorem}

\begin{remark} The last property of the group $T$ is not stated explicitly in \cite{GeneralPink}, but it is clear from the construction (see \cite[Proof of theorem 3.1 assuming theorem 3.2]{GeneralPink}).
\end{remark}

\begin{theorem}\label{thm_PinkGL222}{(\cite[Theorem 4.1]{GeneralPink})}
Let $G$ be a closed subgroup of $\operatorname{GL}_2(\mathbb{Z}_2) \times \operatorname{GL}_2(\mathbb{Z}_2)$ whose projection modulo 4 is trivial. Denote by $G_1, G_2$ the two projections of $G$ on the factors $\operatorname{GL}_2(\mathbb{Z}_2)$, and let $n_1 \geq 4, n_2 \geq 4$ be integers such that $G_i$ contains $\mathcal{B}_2(n_i)$. Suppose furthermore that for every $(g_1,g_2) \in G$ we have $\det(g_1)=\det(g_2)\equiv 1 \pmod 8$.
If $L(G)$ contains $2^k \mathfrak{sl}_2(\mathbb{Z}_2) \oplus 2^k \mathfrak{sl}_2(\mathbb{Z}_2)$ for a certain $k \geq 2$, then $G$ contains
\[
\mathcal{B}_2(12(k+11n_2+5n_1+12)+1,12(k+11n_1+5n_2+12)+1).\]
\end{theorem}

\medskip

Finally, we shall need the following simple lemma regarding conjugation-stable subalgebras of $\mathfrak{sl}_2(\mathbb{Z}_\ell)$:

\begin{lemma}\label{lemma_ConjStableSubspaces}{(\cite[Lemma 2.4]{GeneralPink})}
Let $t$ be a non-negative integer. Let $W \subseteq \mathfrak{sl}_2(\mathbb{Z}_\ell)$ be a Lie subalgebra that does not reduce to zero modulo $\ell^{t+1}$. Suppose that $W$ is stable under conjugation by $\mathcal{B}_\ell(s)$ for some non-negative integer $s$, where $s \geq 2$ if $\ell=2$ and $s \geq 1$ if $\ell=3$ or $5$. Then $W$ contains the open set $\ell^{t+4s+4v} \mathfrak{sl}_2(\mathbb{Z}_\ell)$.
\end{lemma}

\section{The automorphisms of $\mathfrak{sl}_2(\mathbb{Z}_\ell)$ are inner}\label{sect_AutomorphismsAreInner}
In this section we obtain a description of the automorphisms of $\mathfrak{sl}_2(\mathbb{Z}_\ell)$ which shows that -- in a suitable sense -- they are all inner. In order to establish the required result we first need a few simple preliminaries, starting with the following well-known version of Hensel's lemma:

\begin{lemma}\label{lemma_Hensel}
Let $p(x) \in \mathbb{Z}_\ell[x]$ be a monic polynomial and let $\alpha$ be an element of $\mathbb{Z}_\ell$. Suppose that $v_\ell(p(\alpha)) > 2v_\ell(p'(\alpha))$: then $p(x)$ admits a root $\bar{\alpha} \in \mathbb{Z}_\ell$ such that $v_\ell(\alpha-\bar{\alpha}) \geq v_\ell(p(\alpha))-v_\ell(p'(\alpha))$.
\end{lemma}

Here is the main tool we will use to produce approximate roots of polynomials:

\begin{lemma}\label{lemma_ApproximateEigenvalue1}
Let $\ell$ be a prime number, $n \geq 1, m \geq 1$, $g \in \operatorname{End} \left(\mathbb{Z}_\ell^m\right)$ and $p_g(t)$ the characteristic polynomial of $g$. Let furthermore $\lambda \in \mathbb{Z}_\ell, w \in \mathbb{Z}_\ell^m$ be such that $gw \equiv \lambda w \pmod{\ell^n}$. Suppose that at least one of the coordinates of $w$ has $\ell$-adic valuation at most $b$: then $p_g(\lambda) \equiv 0 \pmod{\ell^{n-b}}$.
\end{lemma}

\begin{proof}
Denote by $(g-\lambda \operatorname{Id})^*$ the adjugate matrix of $(g-\lambda \operatorname{Id})$, that is the operator such that $(g-\lambda \operatorname{Id})^*(g-\lambda \operatorname{Id})=\det (g-\lambda \operatorname{Id}) \cdot \operatorname{Id}$. Multiplying the congruence $(g-\lambda \operatorname{Id})w \equiv 0 \pmod{\ell^n}$ on the left by $(g-\lambda \operatorname{Id})^*$ we obtain $\det(g-\lambda \operatorname{Id}) \cdot \operatorname{Id} w \equiv 0 \pmod{\ell^n}$, and by considering the coordinate of $w$ of smallest valuation we deduce $p_g(\lambda)=\det(g-\lambda \operatorname{Id}) \equiv 0 \pmod{\ell^{n-b}}$ as claimed.
\end{proof}


The properties of the adjoint representation of $\mathfrak{sl}_2$ (or an immediate computation) also show:
\begin{lemma}\label{lemma_CommutatorEigenvalues}
Let $g \in \mathfrak{sl}_2(\mathbb{Z}_\ell)$. The linear operator $\mathcal{C}_g:=[g,\cdot]$ from $\mathfrak{sl}_2 (\mathbb{Z}_\ell)$ to itself has eigenvalues $0,\pm 2\mu$, where $\pm \mu$ are the eigenvalues of $g$, so $p_{\mathcal{C}_g}(t)=t(t^2-4\mu^2)$.
\end{lemma}

Combining the previous results we obtain the following lemma, which will be very useful for our purposes:
\begin{lemma}\label{lemma_HenselsFailure}
Let $g$ be an element of $\mathfrak{sl}_2 (\mathbb{Z}_\ell)$, $w$ be a vector in $\mathbb{Z}_\ell^2$, and $\beta$ be the minimal valuation of the coefficients of $w$. Suppose $g w \equiv \lambda w \pmod{\ell^n}$. Then either $g$ has an eigenvalue $\nu$ such that $v_\ell(\nu-\lambda) \geq v_\ell(\lambda)+3$ or else $\beta$ is at least $n-2(2+v_\ell(\lambda))$.
\end{lemma}

\begin{proof}
Let $\pm \mu$ be the eigenvalues of $g$.
From lemma \ref{lemma_ApproximateEigenvalue1} we deduce that $v_\ell(p_g(\lambda)) \geq n-\beta$; notice further that $p_g(t)=t^2-\mu^2$, so $p_g'(t)=2t$. Suppose that $\beta < n-2(2+v_\ell(\lambda))$: then $n-\beta>2(2+v_\ell(\lambda))>2v_\ell(p_g'(\lambda))$, and by Hensel's lemma $p_g(t)$ has a root $\nu$ such that 
\[v_\ell(\nu-\lambda) \geq v_\ell (p_g(\lambda)) - v_\ell(p'_g(\lambda)) \geq n-\beta-v-v_\ell(\lambda) \geq v_\ell(\lambda)+3.\]
\end{proof}


We now come to the central result of this section, which as anticipated is essentially a description of the Lie algebra automorphisms of (the finite quotients of) $\mathfrak{sl}_2(\mathbb{Z}_\ell)$. 

%

\smallskip

\noindent\textbf{Notation.} For the remainder of this section, in order to make notation lighter, when $a$ is a positive integer we write $x=y+O(a)$ for $x \equiv y \pmod{\ell^a}$.


\begin{proposition}\label{prop_ApproximateMorphismsAreInner}
Let $L_1$ be a subalgebra of $\mathfrak{sl}_2(\mathbb{Z}_\ell)$ and $n \geq 1, s \geq 0$ be integers. Suppose that $L_1$ contains $\ell^s\mathfrak{sl}_2(\mathbb{Z}_\ell)$ and that $\varphi: L_1 \to \mathfrak{sl}_2(\mathbb{Z}_\ell)$ is a linear map such that
\[
(\ast) \quad [\varphi(a),\varphi(b)] \equiv \varphi([a,b]) \pmod{\ell^n} \quad \forall a,b \in \ell^s\mathfrak{sl}_2(\mathbb{Z}_\ell).
\]

Define
\[
x=\varphi\left(\ell^s \cdot \left(\begin{matrix} 0 & 1 \\ 0 & 0 \end{matrix} \right)\right), \; y=\varphi\left(\ell^s \cdot \left(\begin{matrix} 0 & 0 \\ 1 & 0 \end{matrix} \right)\right), \; h=\varphi\left(\ell^s \cdot \left(\begin{matrix} 1 & 0 \\ 0 & -1 \end{matrix} \right)\right)
\]
and let $\alpha$ be the minimal integer such that $x,y$ are both nonzero modulo $\ell^{\alpha+1}$.
Suppose that $n \geq \alpha+10s+5v+6$.
There exists a matrix $M \in M_2\left(\mathbb{Z}_\ell\right)$, at least one of whose coefficients is nonzero modulo $\ell$, such that for every $w \in \left(\mathbb{Z}_\ell\right)^2$ and every $g_1 \in L_1$ we have
\begin{equation}\label{eqn_InnerMorphisms}
M(g_1\cdot w) \equiv \varphi(g_1) \cdot M(w) \pmod{\ell^{n-\alpha-6s-4v-6}}.
\end{equation}
Furthermore, $\det(M)$ does not vanish modulo $\ell^{4s+v}$, and for every $g_1$ in $L_1$ we have
\[
\operatorname{tr} \left( \varphi(g_1)^2 \right) \equiv \operatorname{tr} \left(g_1^2\right) \pmod{\ell^{n-\alpha-10s-5v-6}}
\]
and
\[
\varphi(g_1) \equiv Mg_1M^{-1} \pmod{\ell^{n-\alpha-10s-5v-6}}, \quad M^{-1}\varphi(g_1)M \equiv g_1 \pmod{\ell^{n-\alpha-10s-5v-6}}
\]
\end{proposition}

\begin{remark}
A moment's thought (considering the limiting cases when $s$ or $\alpha$ become very large) will reveal that it is indeed necessary for all three parameters $n$, $s$, and $\alpha$ to appear in the conclusion of the proposition. 
The question of whether the \textit{dependence} on these parameters is optimal, on the other hand, is far more complicated, and there is almost certainly room for improvement.
\end{remark}


Here again let us say a few words about the method of proof before fiddling with the technical details. To simplify matters, consider the algebra $L=\mathfrak{sl}_2(\mathbb{Q}_\ell)$. Proving that every automorphism of $L$ is inner basically boils down to showing that the only 2-dimensional representation of $\mathfrak{sl}_2(\mathbb{Q}_\ell)$ is the standard one, a result which is usually proved through the `highest weight vector' machinery: one shows that it is possible to choose an eigenvector $v$ for $h$ that is killed by $x$, and then describes its full orbit under the action of $x,y,h$. More precisely, one shows that $yv$ is an eigenvector for $h$, that $xyv$ is proportional to $v$, and that $y^2v=0$.

The proof that follows mimics this very argument by producing a vector $v_+$, by definition an eigenvector for $h$, which plays the role of the highest weight vector, and subsequently finding its orbit under the action of $h,x,y$. The main difficulty lies in the initial step, where we need to prove that the eigenvalues of $h$ lie in $\mathbb{Z}_\ell$ and are of a certain form. Once this is done, most of the proof looks very much like the one for $\mathfrak{sl}_2(\mathbb{Q}_\ell)$, with the additional complication that we have to keep track of valuations along the way.

\medskip

\begin{proof}
Denote by $\mathcal{C}_{h}$ the linear endomorphism of $\mathfrak{sl}_2(\mathbb{Z}_\ell) \cong \mathbb{Z}_\ell^3$ given by taking the commutator with $h$. It is clear that 
\[
\mathcal{C}_{h}(x) = [h,x] \equiv \varphi \left[\ell^s \cdot \left(\begin{matrix} 1 & 0 \\ 0 & -1 \end{matrix} \right), \ell^s \cdot \left(\begin{matrix} 0 & 1 \\ 0 & 0 \end{matrix} \right) \right] \equiv \varphi\left( 2\ell^s \cdot \ell^s \cdot \left(\begin{matrix} 0 & 1 \\ 0 & 0 \end{matrix} \right) \right) \equiv 2\ell^s x \pmod{\ell^n},
\]
so $x$ is an (approximate) eigenvector of $\mathcal{C}_h$ associated with the (approximate) eigenvalue $2\ell^s$. Lemma \ref{lemma_ApproximateEigenvalue1} yields
\[
p_{\mathcal{C}_h}(2\ell^s) \equiv 0 \pmod{\ell^{n-\alpha}}.
\]

If we let $\pm \mu$ denote the eigenvalues of $h$, then $p_{\mathcal{C}_h}'(t)=(t^2-4\mu^2)+2t^2$, and evaluating at $2\ell^s$ we find
\[
p_{\mathcal{C}_h}'(2\ell^s)=4(\ell^{2s}-\mu^2)+8\ell^{2s}=\frac{p_{\mathcal{C}_h}(2\ell^s)}{2\ell^s}+8\ell^{2s}.
\]

To estimate the $\ell$-adic valuation of this last expression simply observe that
\[
v_\ell\left( \frac{p_{\mathcal{C}_h}(2\ell^s)}{2\ell^s} \right) = v_\ell \left( p_{\mathcal{C}_h}(2\ell^s) \right) - v_\ell(2)-s \geq n-\alpha-v-s >  3v+2s=v_\ell(8\ell^{2s}),
\]
so $v_\ell\left(p_{\mathcal{C}_h}'(2\ell^s)\right)=v_\ell \left(8\ell^{2s}\right)=3v+2s$. By Hensel's lemma (lemma \ref{lemma_Hensel}), $p_{\mathcal{C}_h}(t)$ admits a root $\lambda \in \mathbb{Z}_\ell$ such that
\[
v_\ell(\lambda - 2\ell^s) \geq v_\ell(p_{\mathcal{C}_h}(2\ell^s)  ) - v_\ell(p_{\mathcal{C}_h}'(2\ell^s)) \geq n-\alpha-2s-3v > 2s+1.
\]

Note that $\lambda$ cannot be zero, because clearly $v_\ell(0 - 2\ell^s)=v+s$ is strictly smaller than $v_\ell(\lambda-2\ell^s)$. It follows that $\lambda$ is one of the other two roots of $p_{\mathcal{C}_h}(t)$, namely $\pm 2 \mu$, and hence
\[
\pm \mu = \pm \frac{1}{2} \left( 2\ell^s + O(n-\alpha-2s-3v) \right) = \pm \ell^s(1+O(n-\alpha-3s-4v)).
\]
%
To sum up, the two eigenvalues of $h$ belong to $\mathbb{Z}_\ell$ and are of the form $\pm \ell^s + O(n-\alpha-2s-4v)$, and in particular of the form $\pm \ell^s +O(s+4)$. Let $\mu_{+}$ be the one of the form $\ell^s+O(n-\alpha-2s-4v)$ and $v_{+} \in \mathbb{Z}_\ell^2$ be a corresponding eigenvector, normalized in such a way that at least one of the two coordinates is an $\ell$-adic unit. Set furthermore $v_-=yv_{+}$.

As anticipated, our next objective is to describe the action of $x,y,h$ on $v_{\pm}$. We expect $v_+$ to be annihilated by $x$ and $v_-$ to be an eigenvector for $h$ that is annihilated by $y$: of course this is not going to be exactly true at all orders, but only up to a certain error term that depends on $n$, $\alpha$ and $s$.
Let $\beta$ be the minimal valuation of the coordinates of $xv_+$: this is a number we want to show to be large.

The idea is that if $xv_+$ were not very close to zero, then it would be an eigenvector of $h$ associated with an eigenvalue that $h$ does not possess.  Note that \[h(xv_+) \equiv [h,x]v_{+}+xh v_{+} \equiv (2\ell^s+\mu_+)xv_+ \pmod{\ell^n},\] so by lemma \ref{lemma_HenselsFailure} either $h$ has an eigenvalue $\xi$ such that $v_\ell(\xi-(\mu_++2\ell^s)) \geq 3+v_\ell(\mu_++2\ell^s) \geq s+3$ or $\beta \geq n-2(2+v_\ell(\mu_++2\ell^s))$.
Note now that we cannot be in the first case: indeed $h$ would have an eigenvalue of the form $3\ell^s+O(s+3)$, but we have already seen that the eigenvalues of $h$ are $\pm \ell^s + O(s+4)$, contradiction. Hence we are in the second situation, and furthermore $v_\ell(\mu_++2\ell^s) \leq s+1$: hence $\beta \geq n-2(2+v_\ell(\mu_++2\ell^s)) \geq n-2s-6$, and by definition of $\beta$ this means $xv_{+} \equiv 0 \pmod{\ell^{n-2(s+3)}}$. 
Next we compute
\begin{equation}\label{eq_ActionhOnVMinus}
\begin{aligned}
hv_{-} & = hyv_{+} \\
       & =[h,y]v_{+}+yhv_{+} \\
			 & =-2\ell^s\cdot yv_{+}+y(\mu_{+}v_{+}) +O(n) \\
			 & =(\mu_{+}-2\ell^s)v_{-} +O(n) \\
			 & = (-\ell^s+O(n-\alpha-2s-4v))v_{-} + O(n) \\
			 & = -\ell^s v_{-} +O(n-\alpha-2s-4v),
\end{aligned}
\end{equation}
\begin{equation}\label{eq_ActionOnVMinus}
\begin{aligned}
xv_{-} & = xyv_{+} \\
       & =[x,y]v_{+}+yxv_{+} \\
			 & =\ell^s hv_{+}+O(n-2(s+3)) \\
			 & =\ell^s \mu_{+}v_{+} + O(n-2(s+3)) \\
			 & =\ell^s \left( \ell^s+O(n-\alpha-2s-4v) \right) v_{+} + O(n-2(s+3)) \\
			 & =\ell^{2s}v_{+} + O(n-\alpha-2(s+3));
\end{aligned}
\end{equation}
this settles the question of the action of $h$ and $x$ on $v_-$. We are left with showing that $v_-$ is (approximately) killed by $y$. Again, we do this by showing that $yv_-$ -- unless it is very close to 0 -- yields an eigenvector of $h$ associated with an eigenvalue that $h$ does not possess:
\[
\begin{aligned}
h \cdot yv_{-}& =[h,y]v_{-}+yhv_{-} \\
              & =-2\ell^s \cdot yv_{-}+ y \left( (-\ell^s) +O(n-\alpha-2s-4v) \right) v_{-} + O(n)\\
							& =-3\ell^s yv_{-} +O(n-\alpha-2s-4v),
\end{aligned}
\]
so that $yv_{-}$ is an (approximate) eigenvector of $h$, associated with the (approximate) eigenvalue $-3\ell^s$. Let $\gamma$ be minimal among the valuations of the coefficients of $yv_{-}$. Apply lemma \ref{lemma_HenselsFailure}: either $\gamma \geq n-\alpha-2s-4v-2(2+v_\ell(-3\ell^s)) \geq n-\alpha-4s-4v-6$ or $h$ has an eigenvalue $\nu$ satisfying $v_\ell(\nu+3\ell^s) \geq v_\ell(-3\ell^s)+3 \geq s+3$. This second possibility contradicts what we have already proven on the eigenvalues of $h$, hence $\gamma \geq n-\alpha-4s-4v-6$, that is to say $yv_- = O(n-\alpha-4s-4v-6)$.

%
%
Putting all together, we have proved that up to an error of order $\ell^{n-\alpha-4s-4v-6}$ we have
\[
xv_{+}=0, \; yv_+=v_{-}, \; hv_{+}=\ell^s v_+, \; xv_{-}=\ell^{2s} v_{+}, \; yv_{-}=0, \; hv_{-}=-\ell^s v_{-}.
\]

\medskip

Write $\overline{x}$ (resp. $\overline{y}, \overline{h}$) for $\ell^s \left( \begin{matrix} 0 & 1 \\ 0 & 0 \end{matrix} \right)$ (resp. $\ell^s \left( \begin{matrix} 0 & 0 \\ 1 & 0 \end{matrix} \right), \ell^s \left( \begin{matrix} 1 & 0 \\ 0 & -1 \end{matrix} \right)$) and consider the matrix $\tilde{M}$ whose columns are given by $\ell^s v_{+}$ and $v_-$. The above relations may be stated more compactly as
\begin{equation}\label{eq_MCommutesPhi}
\tilde{M} \overline{x} =x\tilde{M}, \; \tilde{M} \overline{y} =y\tilde{M}, \; \tilde{M} \overline{h} =h\tilde{M}
\end{equation}
modulo $\ell^{n-\alpha-4s-4v-6}$. Let $\delta$ be minimal among the valuations of the coefficients of $\tilde{M}$: by construction, at least one of the coordinates of $v_+$ is an $\ell$-adic unit, so $\delta \leq s$. Set $M=\ell^{-\delta}\tilde{M}$. Dividing equations \eqref{eq_MCommutesPhi} by $\ell^\delta$ we see that $M$ satisfies analogous equations up to error terms of order $n-\alpha-5s-4v-6$, and by construction at least one of the coefficients of $M$ is an $\ell$-adic unit. 

Let now $g$ be any element of $L_1$. The matrix $\ell^s g$ belongs to $\ell^s \mathfrak{sl}_2(\mathbb{Z}_\ell)$, so it is a linear combination of $\overline{x},\overline{y},\overline{h}$ with coefficients in $\mathbb{Z}_\ell$. Write $\ell^s g=\lambda_1 \overline{x} + \lambda_2 \overline{y}+\lambda_3 \overline{h}$. We have
\[
\begin{aligned}
\ell^s M g   & =M(\ell^s g) \\
             & =M(\lambda_1 \overline{x} + \lambda_2 \overline{y}+\lambda_3 \overline{h}) \\
             & =(\lambda_1 x + \lambda_2 y + \lambda_3 h)M + O(n-\alpha-5s-4v-6) \\
						 & = \varphi(\ell^s g)M + O(n-\alpha-5s-4v-6) \\
						 & = \ell^s \varphi(g)M + O(n-\alpha-5s-4v-6),
\end{aligned}
\]
so that dividing by $\ell^s$ we deduce $Mg=\varphi(g)M + O(n-\alpha-6s-4v-6)$ for every $g \in L_1$, which is the first statement in the proposition.

\medskip

Let us now turn to the statement concerning the determinant. We can assume that $v_+$ is normalized so that $v_+=\left( \begin{matrix} 1 \\ c \end{matrix} \right)$. We also write $v_-=\left(\begin{matrix} b \\ d \end{matrix} \right)$. It is clear that $v_\ell(\det M) \leq v_\ell(\det \tilde{M})$, and that  $\det \tilde{M} = \ell^s \det \left( \begin{matrix} 1 & b \\ c & d \end{matrix} \right)$, so let us consider $D:=v_\ell\left(\det \left( \begin{matrix} 1 & b \\ c & d \end{matrix} \right)\right)$. Suppose by contradiction $D>3s+v$; by definition of the determinant we have $d=bc+O(D)$, which implies
\[
v_-=\left( \begin{matrix} b \\ d \end{matrix} \right) = \left( \begin{matrix} b \\ bc+O(D) \end{matrix} \right) = bv_++O(D).
\]

Applying $h$ to both sides of this equality and using equation \eqref{eq_ActionhOnVMinus} we get
\[
\mu_- v_- + O(n-\alpha-2s-4v) = hv_-=h(bv_++O(D))=b\mu_+v_+ + O(D).
\]

Comparing the first coordinate of these vectors we deduce
\[
b\mu_-= b\mu_+ +O(\min\left\{D,n-\alpha-2s-4v \right\}),
\]
hence
\begin{equation}\label{eq_IncompatibleEigenvalues}
\mu_-=\mu_+ + O(\min\left\{D-v_\ell(b),n-\alpha-2s-4v-v_\ell(b)\right\}).
\end{equation}

Note now that since $d=bc+O(D)$ we have $v_\ell(d) \geq \min\left\{v_\ell(b),D \right\}$. Moreover, we see by equation \eqref{eq_ActionOnVMinus} that $xv_-=\ell^{2s}v_++O(n-\alpha-2(s+3))$, and since the right hand side does not vanish modulo $\ell^{2s+1}$ (since $n-\alpha-2(s+3) >2s+1$ and $\ell^{2s}v_+=\left( \begin{matrix} \ell^{2s} \\ \ell^{2s}c \end{matrix} \right)$) we must in particular have $v_- \not \equiv 0 \pmod{\ell^{2s+1}}$. Hence $\min\left\{v_\ell(b),v_\ell(d)\right\} \leq 2s$. Let us show that we also have $v_\ell(b) \leq 2s$. Suppose that $v_\ell(b) \geq 2s+1$: then
\[
v_\ell(d) \geq \min\left\{v_\ell(b),D \right\} \geq \min\left\{2s+1,3s+v+1 \right\} \geq 2s+1,
\]
which implies $\min\left\{v_\ell(b),v_\ell(d)\right\} \geq 2s+1$ and contradicts what we just proved. 

Therefore $v_\ell(b) \leq 2s$, hence equation \eqref{eq_IncompatibleEigenvalues} implies $\mu_-=\mu_+ + O \left(D-2s\right)$: notice that if the minimum in \eqref{eq_IncompatibleEigenvalues} were attained for $n-\alpha-2s-4v-v_\ell(b) \geq 3s+2$ we would have $\ell^s = -\ell^s + O(3s+2)$, a clear contradiction. 
On the other hand, we know that $\mu_\pm = \pm \ell^s + O(s+4)$, so the above equation implies $2\ell^s+O(s+4)=O(D-2s)$. Hence we have proved $v_\ell(2\ell^s) \geq D-2s$, i.e. $D \leq 3s+v$, a contradiction. It follows, as claimed, that $v_\ell(\det M) \leq v_\ell(\det \tilde{M})=s+D \leq 4s+v$.

Next we prove the statement concerning traces. Let $g$ be any element of $L_1$. Setting, for the sake of simplicity, $N=n-\alpha-6s-4v-6$, we have $Mg=\varphi(g)M + O(N)$, so (multiplying on the left by the adjugate $M^*$ of $M$) we deduce $\det(M) g = M^* \varphi(g) M +O(N)$. Didiving through by $\det(M)$ we have $g = M^{-1} \varphi(g) M + O(N-(4s+v))$; note that this equality would a priori only hold in $\mathfrak{sl}_2(\mathbb{Q}_\ell)$, but since both $g$ and the error term are $\ell$-integral we necessarily also have $M^{-1} \varphi(g) M \in \mathfrak{sl}_2(\mathbb{Z}_\ell)$. Squaring and taking traces then yields $\operatorname{tr}\left(g^2\right) = \operatorname{tr} \left[ \left( M^{-1} \varphi(g) M \right)^2 \right] + O(N-(4s+v))$, i.e.
\[
\operatorname{tr}\left(g^2\right) = \operatorname{tr} \left(\varphi(g)^2 \right)+ O(N-(4s+v))
\]
as claimed. Finally, essentially the same argument shows the last two statements: we can multiply the congruence $Mg_1 \equiv \varphi(g_1) M \pmod{\ell^{N}}$ on the right (resp.~left) by $M^*$ and divide by $\det M$ to get
\[
Mg_1M^{-1} \equiv \varphi(g_1) \pmod{\ell^{N-4s-v}}, \quad g_1 \equiv M^{-1}\varphi(g_1)M \pmod{\ell^{N-4s-v}}.\]\end{proof}

\section{Products of two curves}\label{sect_ProductOfTwo}
\subsection{Notation and preliminaries}
Let $E_1, E_2$ be two elliptic curves over $K$ and $\ell$ be a prime number, and recall that we denote by $G_\ell$ the image of $\abGal{K}$ inside $\operatorname{Aut} T_\ell(E_1) \times \operatorname{Aut} T_\ell(E_2) \cong \operatorname{GL}_2(\mathbb{Z}_\ell)^2$.
To study the Galois representation attached to $E_1 \times E_2$ we are going to pass to a suitable extension of $K$ over which the study of the Lie algebra of $G_\ell$ is sufficient to yield information on $G_\ell$ itself. Before doing this, however, we need to dispense with some necessary preliminaries. Let $G_{\ell,1}, G_{\ell,2}$ be the two projections of $G_\ell$ onto the two factors $\operatorname{GL}_2(\mathbb{Z}_\ell)$, and $m_1$, $m_2$ be integers such that $\mathcal{B}_\ell(m_i)$ is contained in $G_{\ell,i}$ for $i=1,2$. 
We want to apply theorem \ref{thm_PinkGL22}, so for the whole section (with the exception of proposition \ref{prop_Finale2Curves}) we make the following

\smallskip

\begin{assumption}\label{assumpt_GaloisIsSmall} If $\ell$ is odd, $G_\ell$ does not contain $\mathcal{B}_\ell\left(20\max\{m_1,m_2\},20\max\{m_1,m_2\} \right)$.
\end{assumption}

\smallskip

Under this assumption, we define $K_\ell$ to be the extension of $K$ associated with the following closed subgroups of $G_\ell$:
\[
\begin{cases}
\operatorname{ker}\left(G_2 \to \operatorname{GL}_2(\mathbb{Z}/8\mathbb{Z})^2 \right), \text{ if } \ell=2 \\
H_\ell, \text{ if } \ell \neq 2,
\end{cases}
\]
where $H_\ell$ is the group given by an application of theorem \ref{thm_PinkGL22} under our assumption. Note that the degree $[K_2:K]$ is at most $3^2 2^{16}$, that is to say the order of
\[
\left\{(x,y) \in \operatorname{GL}_2(\mathbb{Z}/8\mathbb{Z})^2 \bigm \vert \det x= \det y \right\},
\]
whereas $[K_\ell:K]$ is uniformly bounded by $2 \cdot 48^2$ for $\ell \neq 2$.
Note that $H_\ell$ is by construction the image of $\abGal{K_\ell}$ in $\operatorname{Aut} T_\ell(E_1) \times \operatorname{Aut} T_\ell(E_2) \cong \operatorname{GL}_2(\mathbb{Z}_\ell)^2$. 

\begin{definition}\label{def_n1n2}
We write $H_{\ell,1}, H_{\ell,2}$ for the projections of $H_\ell$ on the two factors $\operatorname{GL}_2(\mathbb{Z}_\ell)$. Furthermore, we let $n_1,n_2$ be integers such that $H_{\ell,1}, H_{\ell,2}$ respectively contain $\mathcal{B}_\ell(n_1),\mathcal{B}_\ell(n_2)$. 
Notice that if $\ell=2$ we have $n_1, n_2 \geq 2$; on the other hand, for $\ell=3$ or $5$ we explicitly demand that $n_1, n_2 \geq 1$. 
\end{definition}



\begin{remark}\label{rmk_ValuesofMandN}
Note that if $m_1, m_2>0$ we can in fact take $n_1=m_1,n_2=m_2$ unless $\ell \leq 3$: indeed for primes $\ell \geq 5$ the index of $H_\ell$ in $G_\ell$ is not divisible by $\ell$, so for any positive value of $n$ the (pro-$\ell$) group $\mathcal{B}_\ell(n)$ is contained in $H_{\ell,i}$ if and only if it is contained in $G_{\ell,i}$. Furthermore, it is clear by definition that we can always assume without loss of generality $m_1 \leq n_1, m_2 \leq n_2$, because the groups $H_{\ell,i}$ are subgroups the corresponding groups $G_{\ell,i}$.
\end{remark}


\begin{definition}
We let $L \subseteq \mathfrak{sl}_2(\mathbb{Z}_\ell)^{\oplus 2}$ (resp.~$L_1, L_2 \subseteq \mathfrak{sl}_2(\mathbb{Z}_\ell)$) be the Lie algebra of $H_\ell$ (resp.~of $H_{\ell,1}$, of $H_{\ell,2}$). We choose a basis of $L$ of the form $(a_1,b_1), (a_2,b_2), (a_3,b_3),(0,y_1),(0,y_2),(0,y_3)$. Such a basis clearly exists. 
\end{definition}

By our assumption $H_{\ell,1} \supseteq \mathcal{B}_\ell(n_1)$ we have $L_1 \supseteq \ell^{n_1}\mathfrak{sl}_2(\mathbb{Z}_\ell)$.
Notice that $(0,y_1),(0,y_2),(0,y_3)$ span a Lie-subalgebra: indeed $[(0,y_i),(0,y_j)]=(0,[y_i,y_j])$ must be a linear combination with $\mathbb{Z}_\ell$-coefficients of the basis elements; however, since $a_1,a_2,a_3$ are linearly independent over $\mathbb{Z}_\ell$, we deduce that this commutator is a linear combination of $(0,y_1),(0,y_2),(0,y_3)$, so that these three elements do indeed span a Lie algebra, which we call $\ElThree$. Note that $\ElThree$ can equivalently be described as the kernel of the projection from $L \subseteq \mathfrak{sl}_2(\mathbb{Z}_\ell) \oplus \mathfrak{sl}_2(\mathbb{Z}_\ell)$ to the first copy of $\mathfrak{sl}_2(\mathbb{Z}_\ell)$. We shall interchangeably consider $\ElThree$ as a subalgebra of $\mathfrak{sl}_2(\mathbb{Z}_\ell) \oplus \mathfrak{sl}_2(\mathbb{Z}_\ell)$ or as a subalgebra of $\mathfrak{sl}_2(\mathbb{Z}_\ell)$, identifying $\ElThree$ to its projection on the second copy of $\mathfrak{sl}_2(\mathbb{Z}_\ell)$.

\begin{lemma}
$\ElThree \subseteq \mathfrak{sl}_2(\mathbb{Z}_\ell)$ is stable under conjugation by $\mathcal{B}_\ell(n_2)$.
\end{lemma}

\begin{proof}
For the proof, consider $\ElThree$ as a subalgebra of $\mathfrak{sl}_2(\mathbb{Z}_\ell) \oplus \mathfrak{sl}_2(\mathbb{Z}_\ell)$.
Take any element $l \in \ElThree$: it is the limit of a sequence $l_n=\sum_{i=1}^n \lambda_{n,i} \Theta(g_{n,i})$ for certain $g_{n,i} \in H_\ell$ and $\lambda_{n,i} \in \mathbb{Z}_\ell$. For any $g \in B_\ell(n_2)$ there exists $h \in H_{\ell,1}$ such that $(h,g)$ is in $H_\ell$. We have
\[
\begin{aligned}
(h,g)^{-1}l_n(h,g) & =\sum_{i=1}^n \lambda_{n,i} (h,g)^{-1}\Theta(g_{n,i})(h,g)  = \sum_{i=1}^n \lambda_{n,i} (h,g)^{-1}\left(g_{n,i}-\frac{\operatorname{tr}(g_{n,i})}{2} \operatorname{Id}\right)(h,g) \\
				 & = \sum_{i=1}^n \lambda_{n,i} \left((h,g)^{-1}g_{n,i}(h,g)-\frac{\operatorname{tr}((h,g)^{-1}g_{n,i}(h,g))}{2} \operatorname{Id}\right) \\
				 & =\sum_{i=1}^n \lambda_{n,i} \Theta((h,g)^{-1}g_{n,i}(h,g)) \in \langle \Theta(H_\ell) \rangle,
\end{aligned}
\]
so the sequence $\left((h,g)^{-1}l_n(h,g)\right)_{n \geq 0}$ is in $L$, and by continuity of conjugation tends to the element $(h,g)^{-1}l(h,g)$ of $L$. Now if we write $l=(l^{(1)},l^{(2)})=(0,l^{(2)})$ we have
\[(h,g)^{-1}l(h,g)=(h,g)^{-1} (0,l^{(2)}) (h,g)=(0,g^{-1}l^{(2)}g) \in L,\]
and since $\ElThree$ is exactly the sub-algebra given by the elements whose first coordinate vanishes the claim is proved.
\end{proof}

\begin{corollary}\label{cor_EasyCase}
Fix an integer $t$, and suppose that at least one among $y_1,y_2,y_3$ is nonzero modulo $\ell^{t+1}$: then $\ElThree$ contains $\ell^{t+4n_2+4v}\mathfrak{sl}_2(\mathbb{Z}_\ell)$.
\end{corollary}
\begin{proof}
Apply lemma \ref{lemma_ConjStableSubspaces} with $s=n_2$ (the hypothesis of this lemma are satisfied thanks to our assumptions on $n_1,n_2$, cf.~definition \ref{def_n1n2}).
\end{proof}

Our task is therefore to bound the values of $t$ for which the $y_i$ all vanish modulo $\ell^t$; doing this will allow us to prove proposition \ref{prop_Finale2Curves} below, which is the crucial ingredient in proving theorem \ref{thm_Explicit}. The desired bound on $t$ will be established in §\ref{subsect_ExOpTh}; in the present section we content ourselves with showing two basic lemmas in linear algebra (over $\mathbb{Z}_\ell$) which will be useful later. Notice that -- since we are only looking for an upper bound on $t$ -- there is no loss of generality in assuming that $y_i \equiv 0 \pmod {\ell^{n_2+1}}$ for $i=1,2,3$, for otherwise we are already done. 
Thus we can assume that we are given an integer $t\geq n_2+1$ such that $y_i \equiv 0 \pmod{\ell^t}$ for $i=1,2,3$: we shall endeavour to show an upper bound on the value of $t$. As a first modest step in this direction we have:


\begin{lemma}\label{lemma_bigenerate}
Suppose $y_1, y_2, y_3$ all vanish modulo $\ell^t$, where $t \geq n_2+1$. 
The $\mathbb{Z}_\ell$-submodule of $\mathfrak{sl}_2(\mathbb{Z}_\ell)$ generated by $b_1$, $b_2$, $b_3$ contains $\ell^{n_2} \mathfrak{sl}_2(\mathbb{Z}_\ell)$.
\end{lemma}

It is clear that the previous lemma follows immediately from the following more general statement:
\begin{lemma}\label{lemma_SufficientCondition}
Let $b_1,\ldots,b_k$ and $y_1,\ldots,y_k$ be elements of $\mathbb{Z}_\ell^k$, and let $n$ be a non-negative integer. Suppose that $y_1,\ldots,y_k$ are all zero modulo $\ell^{n+1}$, and that the submodule of $\mathbb{Z}_\ell^k$ generated by $b_1,\ldots,b_k,y_1,\ldots,y_k$ contains $\ell^n \mathbb{Z}_\ell^k$. Then the submodule of $\mathbb{Z}_\ell^k$ generated by $b_1,\ldots,b_k$ contains $\ell^n \mathbb{Z}_\ell^k$. Let furthermore $e_1,\ldots,e_k$ be the standard basis of $\mathbb{Z}_\ell^k$: there exists a $T \in \operatorname{End}_{\mathbb{Z}_\ell}(\mathbb{Z}_\ell^k)$ such that $Tb_i=\ell^ne_i$ for $i=1,\ldots,k$.
\end{lemma}
\begin{proof}
Let $B, Y$ be the $k \times k$ matrices that have the $b_i$ (resp.~the $y_j$) as columns. The hypothesis implies that there exist two matrices $\tilde{B}, \tilde{Y}$ such that $\ell^n \operatorname{Id} = B \tilde{B} + Y \tilde{Y}$. Notice that by assumption $Y$ is zero modulo $\ell^{n+1}$, so we can rewrite this equation as
$
\displaystyle \ell^n \left(\operatorname{Id} - \ell \frac{Y}{\ell^{n+1}} \tilde{Y} \right) = B \tilde{B},
$
where $ \displaystyle \frac{Y}{\ell^{n+1}} $ has $\ell$-integral coefficients. Now observe that $\left(\operatorname{Id} - \ell \frac{Y}{\ell^{n+1}} \tilde{Y} \right)$, being congruent to the identity modulo $\ell$, is invertible in $\operatorname{End}_{\mathbb{Z}_\ell}\left(\mathbb{Z}_\ell^{k} \right)$; it follows that
$
\ell^n \operatorname{Id} = B \tilde{B} \left(\operatorname{Id} - \ell \frac{Y}{\ell^{n+1}} \tilde{Y} \right)^{-1},
$
which gives a representation of the vectors $\ell^n e_i$ as $\mathbb{Z}_\ell$-linear combinations of $b_1,\ldots,b_k$. Finally, since left- and right- inverses of matrices agree, we also have $\ell^n \operatorname{Id}= \tilde{B} \left(\operatorname{Id} - \ell \frac{Y}{\ell^{n+1}} \tilde{Y} \right)^{-1} B$, so for the second statement we can take $T:=\tilde{B} \left(\operatorname{Id} - \ell \frac{Y}{\ell^{n+1}} \tilde{Y} \right)^{-1}$.
\end{proof}

We shall also need the following consequence of lemmas \ref{lemma_bigenerate} and \ref{lemma_SufficientCondition}.

\begin{lemma}\label{lemma_BasisCongruence}
Let $\lambda_1,\lambda_2,\lambda_3 \in \mathbb{Z}_\ell$ and let $n$ be a non-negative integer. Suppose that
\begin{equation}\label{eq_congruence}
\lambda_1 b_1 + \lambda_2 b_2 + \lambda_3 b_3 \equiv 0 \pmod {\ell^{n_2+n}}:
\end{equation}
then $\lambda_1,\lambda_2,\lambda_3$ are all zero modulo $\ell^n$.
\end{lemma}
\begin{proof}Let $e_1=\left(\begin{matrix}0 & 1 \\ 0 & 0 \end{matrix} \right)$, $e_2=\left(\begin{matrix}1 & 0 \\ 0 & -1 \end{matrix} \right)$, $e_3=\left(\begin{matrix}0 & 0 \\ 1 & 0 \end{matrix} \right)$.
By the previous lemma, there is a $T \in \operatorname{End}_{\mathbb{Z}_\ell}(\mathfrak{sl}_2(\mathbb{Z}_\ell))$ such that $Tb_i=\ell^{n_2}e_i$ for $i=1,2,3$. Applying $T$ to both sides of \eqref{eq_congruence} we find $\ell^{n_2}\left(\lambda_1e_1+\lambda_2 e_2+\lambda_3e_3\right) \equiv 0 \pmod{\ell^{n_2+n}}$: this clearly implies that $\lambda_1,\lambda_2,\lambda_3$ are all zero modulo $\ell^{n}$.
\end{proof}

\subsection{An explicit open image theorem for $G_\ell$}\label{subsect_ExOpTh}
Let us now return to our elliptic curves $E_1, E_2$. We continue with the notation from the previous section. 
Notice that $a_1$, $a_2$, $a_3$ are $\mathbb{Z}_\ell$-linearly independent, hence there exists a unique $\mathbb{Z}_\ell$-linear map $\varphi: L_1 \to L_2$ such that $\varphi(a_i)=b_i$ for $i=1,2,3$. 

For two indices $j,k$ write $[a_j,a_k]=\sum_{i=1}^3 \mu_i^{(j,k)} a_i$ for certain structure constants $\mu_i^{(j,k)} \in \mathbb{Z}_\ell$. Recall furthermore that $t$ is a positive integer no less than $n_2+1$ and such that $y_1$, $y_2$, $y_3$ are all zero modulo $\ell^t$.
Since $L$ is a Lie algebra, there exist scalars $\nu_i^{(j,k)}$ such that
\[
\left[ (a_j,b_j),(a_k,b_k) \right]=\sum_{i=1}^3 \mu_i^{(j,k)} (a_i,b_i) + \sum_{i=1}^3 \nu^{(j,k)}_i (0,y_i),
\]
and reducing the second coordinate of this equation modulo $\ell^t$ gives
\[
\begin{aligned}
\phantom{.}[\varphi(a_j),\varphi(a_k)] & =[b_j,b_k]  \equiv \sum_{i=1}^3 \mu_i^{(j,k)} b_i
                             \equiv  \sum_{i=1}^3 \mu_i^{(j,k)} \varphi(a_i) \\
                            & \equiv \varphi \left( \sum_{i=1}^3 \mu_i^{(j,k)} a_i \right)
                             \equiv \varphi\left( [a_j,a_k] \right) \pmod{\ell^t}.
\end{aligned}
\]

We want to apply proposition \ref{prop_ApproximateMorphismsAreInner} to $\varphi$. We claim that, in the notation of that proposition, we can take $\alpha \leq n_2+n_1$. By lemma \ref{lemma_BasisCongruence}, a linear combination $\lambda_1 b_1 + \lambda_2 b_2 + \lambda_3 b_3$ can vanish modulo $\ell^{n_1+n_2+1}$ only if $\lambda_1, \lambda_2, \lambda_3$ all vanish modulo $\ell^{n_1+1}$. 

Now since the $\mathbb{Z}_\ell$-module generated by $a_1,a_2,a_3$ contains $\ell^{n_1}\mathfrak{sl}_2(\mathbb{Z}_\ell)$ we can choose scalars $\lambda_1,\lambda_2,\lambda_3 \in \mathbb{Z}_\ell$ such that $\ell^{n_1} \left( \begin{matrix} 0 & 1 \\ 0 & 0 \end{matrix} \right) =\lambda_1 a_1+\lambda_2 a_2+\lambda_3 a_3$, and clearly at least one among $\lambda_1$, $\lambda_2$, $\lambda_3$ is nonzero modulo $\ell^{n_1+1}$. It follows by lemma \ref{lemma_BasisCongruence} that 
\[
\varphi\left( \ell^{n_1} \left( \begin{matrix} 0 & 1 \\ 0 & 0 \end{matrix} \right)\right) = \varphi\left(\lambda_1 a_1+\lambda_2 a_2+\lambda_3 a_3 \right) = \sum_{i=1}^3 \lambda_i b_i
\]
is nonzero modulo $\ell^{n_1+n_2+1}$ as claimed, and a perfectly analogous argument applies to the image of $\ell^{n_1} \left( \begin{matrix} 0 & 0 \\ 1 & 0 \end{matrix} \right)$. 
We also claim that by construction of $\varphi$ and by our assumption on $t$ we have
\[
(l_1,l_2) \in L(H_\ell) \Rightarrow l_2  \equiv \varphi(l_1) \pmod {\ell^t}.
\]
To see this, recall first that every element $(l_1,l_2) \in L(H_\ell)$ is a linear combination of the $(a_i,b_i)$ (for $i=1,2,3$) and of the $(0,y_j)$ (for $j=1,2,3$). Writing $(l_1,l_2)=\sum_{i=1}^3 \sigma_i (a_i,b_i) + \sum_{j=1}^3 \tau_j (0,y_j)$ for some scalars $\sigma_i, \tau_j \in \mathbb{Z}_\ell$, and using the fact that $y_j \equiv 0\pmod{\ell^t}$ for $j=1,2,3$, we find $l_1=\sum_{i=1}^3\sigma_i a_i$ and 
\[
l_2 = \sum_{i=1}^3 \sigma_i b_i + \sum_{j=1}^3 \tau_j y_j \equiv \sum_{i=1}^3 \sigma_i \varphi(a_i) \equiv \varphi\left( \sum_{i=1}^3 \sigma_i a_i \right)\equiv \varphi(l_1) \pmod{\ell^t}.
\]
Set $T=t-11n_1-n_2-5v-6$. By proposition \ref{prop_ApproximateMorphismsAreInner}, there is a matrix $M \in M_2(\mathbb{Z}_\ell)$ such that:
\begin{enumerate}
\item for all $(l_1,l_2) \in L(H_\ell)$ we have $l_2 \equiv M \cdot l_1 \cdot M^{-1} \pmod{\ell^T}$ and $M^{-1} \cdot l_2 \cdot M \equiv l_1 \pmod{\ell^T}$;
\item at least one of the coefficients of $M$ is an $\ell$-adic unit;
\end{enumerate}
furthermore, the map $\varphi$ satisfies
\begin{enumerate}
\item[3.] $\operatorname{tr} ({l_1}^2) \equiv \operatorname{tr}(\varphi(l_1)^2)  \equiv \operatorname{tr} ({l_2}^2) \pmod{\ell^T} \quad \forall (l_1,l_2) \in L(H_\ell)$.
\end{enumerate}
Take any element $(g_1,g_2) \in H_\ell$. By our choice of $K_\ell$, we know that the determinant of $g_1$ (which is equal to the determinant of $g_2$) is a square in $\mathbb{Z}_\ell^\times$, so we can choose a square root of $\det{g_1}$ in $\mathbb{Z}_\ell$ and write
\begin{equation}\label{eq_DivideByDet}
(g_1,g_2) = \sqrt{\det{g_1}} (g_1',g_2')
\end{equation}
for a certain $(g_1',g_2') \in \operatorname{SL}_2(\mathbb{Z}_\ell)$. Set $(l_1,l_2)=\Theta_2(g_1',g_2')$, and notice that $(l_1,l_2)$ differs from $\Theta_2(g_1,g_2)$ by an invertible scalar factor, so it lies again in $L(H_\ell)$. By definition of $\Theta_2$, there exists a pair $(\lambda_1,\lambda_2) \in \mathbb{Z}_\ell^2$ such that
\begin{equation}\label{eq_lambdas}
(g_1',g_2')=(\lambda_1,\lambda_2) \cdot \operatorname{Id} + \left(l_1,l_2 \right),
\end{equation}
and we wish to show that $\lambda_1$ is congruent to $\lambda_2$ modulo a large power of $\ell$:
\begin{proposition}
We have $\lambda_1 \equiv \lambda_2 \pmod {\ell^{T-2v}}$ and $g_2 \equiv Mg_1M^{-1} \pmod {\ell^{T-2v}}$.
\end{proposition}

\begin{proof}
Notice first that the second statement follows immediately from the first, combined with equations \eqref{eq_DivideByDet} and \eqref{eq_lambdas} and the fact that for all $(l_1,l_2) \in L(H_\ell)$ we have $l_2 \equiv Ml_1M^{-1} \pmod{\ell^T}$ by the properties of $M$. Hence we just need to prove the first congruence.

 We begin by discussing the case of odd $\ell$. Squaring equation \eqref{eq_lambdas} we obtain
\[
\left(\left(g_1'\right)^2,\left(g_2'\right)^2 \right)=({\lambda_1}^2 \cdot \operatorname{Id} + {l_1}^2 + 2 \lambda_1 l_1, {\lambda_2}^2 \cdot \operatorname{Id} + {l_2}^2 + 2 \lambda_2 l_2 ).
\]

Now the left hand side is simply $\displaystyle \frac{1}{\det{g_1}} \left({g_1}^2,{g_2}^2\right)$, an element of $H_\ell$ up to scalar factors. The image of this matrix through $\Theta_2$ is then an element of $L(H_\ell)$, so applying $\Theta_2$ to the right hand side of the previous equation we get
\begin{equation}\label{eq_SquaresInHl}
\left(\Theta_1({l_1}^2)+2\lambda_1 l_1 ,  \Theta_1({l_2}^2) + 2\lambda_2 l_2 \right) \in L(H_\ell),
\end{equation}
which by the properties of $M$ implies
$
\Theta_1({l_2}^2) + 2\lambda_2 l_2 \equiv M \left(\Theta_1({l_1}^2)+2\lambda_1 l_1  \right) M^{-1} \pmod{\ell^T},
$
or equivalently
\begin{equation}\label{eq_Lambdas1}
\displaystyle {l_2}^2-\frac{\operatorname{tr}\left({l_2}^2\right)}{2} \operatorname{Id} + 2\lambda_2 l_2 \equiv M \left(\displaystyle {l_1}^2-\frac{\operatorname{tr}\left({l_1}^2\right)}{2} \operatorname{Id}+2\lambda_1 {l_1}  \right) M^{-1} \pmod{\ell^T}.
\end{equation}
Now since we also have $\frac{1}{2}\operatorname{tr} \left({l_1}^2\right) \equiv \frac{1}{2} \operatorname{tr} \left( {l_2}^2\right) \pmod{\ell^T}$ (recall that $\ell$ is odd, so we can safely divide by 2) and ${l_2}^2 \equiv M {l_1}^2 M^{-1} \pmod{\ell^T}$ we see that \eqref{eq_Lambdas1} implies
\[
2\lambda_1 l_2 \equiv M \left( 2 \lambda_1 l_1 \right) M^{-1} \equiv 2 \lambda_2 l_2 \pmod{\ell^T}.
\]

If $l_2$ has at least one coordinate not divisible by $\ell$, this last equation implies $\lambda_1 \equiv \lambda_2 \pmod{\ell^{T}}$. If not, then $g_2'$ reduces modulo $\ell$ to a multiple of the identity (cf.~equation \eqref{eq_lambdas}). Moreover, as $\det(g_2')=1$, we have in particular
\[
1=\det(\lambda_2 \operatorname{Id} + l_2) = {\lambda_2}^2 - \frac{\operatorname{tr}\left({l_2}^2\right)}{2},
\]
from which we find
\begin{equation}\label{eq_lambda2}
\lambda_2 = \pm \sqrt{\displaystyle 1+\frac{\operatorname{tr}({l_2}^2)}{2}},
\end{equation}
where the square root can be computed via the usual series expansion $\displaystyle \sqrt{1+t}=\sum_{j \geq 0} {1/2 \choose j} t^j$, which converges because $l_2$ is trivial modulo $\ell$ (hence the same is true for $\operatorname{tr}\left({l_2}^2\right)$).  Symmetrically we prove that either the congruence $\lambda_1 \equiv \lambda_2 \pmod{\ell^{T}}$ holds or else $l_1$ is trivial modulo $\ell$ and 
\begin{equation}\label{eq_lambda1}
\lambda_1 = \pm \sqrt{\displaystyle 1+\frac{\operatorname{tr}({l_1}^2)}{2}}.
\end{equation}

Suppose then $l_1,l_2$ to both be trivial modulo $\ell$: then $\operatorname{tr}({l_1}^2)$ and $\operatorname{tr}({l_2}^2)$ are divisible by $\ell$, hence the square roots appearing in equations \eqref{eq_lambda2} and \eqref{eq_lambda1} can both be computed by the series expansion recalled above. Hence using the congruence $\displaystyle \frac{\operatorname{tr}({l_1}^2)}{2} \equiv \frac{\operatorname{tr}({l_2}^2)}{2} \pmod{\ell^T}$ we easily obtain $\sqrt{\displaystyle 1+\frac{\operatorname{tr}({l_1}^2)}{2}} \equiv \sqrt{\displaystyle 1+\frac{\operatorname{tr}({l_2}^2)}{2}} \pmod{\ell^T}$, hence to prove our claim $\lambda_1 \equiv \lambda_2 \pmod{\ell^T}$ it suffices to show that the signs in equations \eqref{eq_lambda2} and \eqref{eq_lambda1} are the same. 

Since $\sqrt{\displaystyle 1+\frac{\operatorname{tr}({l_1}^2)}{2}} \equiv \sqrt{\displaystyle 1+\frac{\operatorname{tr}({l_2}^2)}{2}} \equiv 1 \pmod \ell$, in order to prove this it is enough to show that $\lambda_1 \equiv \lambda_2 \pmod \ell$.
Now observe that $g_1',g_2'$ reduce to diagonal matrices $\operatorname{diag}\left( \lambda_i,\lambda_i\right)$ in $\operatorname{SL}_2(\mathbb{F}_\ell)$, so we have $\lambda_1 \equiv \lambda_2 \pmod{\ell}$ if and only if $g_1', g_2'$ (or equivalently $g_1, g_2$) have the same reduction modulo $\ell$, and this is exactly one of the properties of $H_\ell$ given by theorem \ref{thm_PinkGL22}. This establishes the claim when $\ell$ is odd.



Consider now the case $\ell=2$. Then $l_1, l_2$ vanish modulo $4$ by definition of $H_2$, and the same argument as above shows that
\begin{equation}\label{eq_Lambdas2}
\lambda_{i} = \pm \sqrt{\displaystyle 1+\frac{\operatorname{tr}({l_{i}}^2)}{2}}, \; i=1,2.
\end{equation}
Given that $2\lambda_{i} \equiv \operatorname{tr} \left( g_{i}' \right) \equiv 2 \pmod 8$ by our construction of $H_2$, we have $\lambda_{1} \equiv \lambda_2 \equiv 1 \pmod 4$, so the sign in equation \eqref{eq_Lambdas2} must be a plus both for $i=1$ and $i=2$. From the congruence $\displaystyle \frac{\operatorname{tr}\left( {l_1}^2 \right)}{2} \equiv \frac{\operatorname{tr}\left( {l_2}^2 \right)}{2} \pmod{2^{T-1}}$ we then deduce
\[
\lambda_1 \equiv \sqrt{1 + \frac{\operatorname{tr}({l_1}^2)}{2}} \equiv \sqrt{1 + \frac{\operatorname{tr}({l_2}^2)}{2}} \equiv \lambda_2 \pmod{2^{T-2}}
\]
as claimed.
\end{proof}
Let us take a moment to summarize what we have proved so far. We have set $\ElThree$ to be the kernel of the natural projection from $L(H_\ell)$, the Lie algebra of $H_\ell$, to $L_1$, the Lie algebra of $L(H_{\ell,1})$. We know that $\ElThree$ is generated by three elements $(0,y_1), (0,y_2), (0,y_3)$, and we have just proved that if $y_1, y_2, y_3$ are all zero modulo $\ell^t$ for a certain positive integer $t \geq n_2+1$, then there exists a matrix $M \in M_2(\mathbb{Z}_\ell)$ with the following properties:
\begin{itemize}
\item at least one of the coefficients of $M$ is an $\ell$-adic unit;
\item for every $(g_1,g_2) \in H_\ell$ we have $g_2 \equiv Mg_1M^{-1} \pmod {\ell^{T-2v}}$.
\end{itemize}

We shall now use these facts to give an upper bound on the values of $t$ for which we can have $y_1 \equiv y_2 \equiv y_3 \equiv 0 \pmod{\ell^t}$: 

\begin{proposition}
Set
$
t_{\max}:=\left\lfloor \frac{v_\ell(b_0(E_1 \times E_2/K_\ell))}{2} \right\rfloor+ 11n_1+n_2+7v+7.
$
At least one of $y_1, y_2, y_3$ does not vanish modulo $\ell^{t_{\max}}$.
\end{proposition}
\begin{proof}
If one of $y_1, y_2, y_3$ is nonzero modulo $\ell^{n_2}$ we are done, so suppose once more that $t \geq n_2+1$ is an integer such that $y_1 \equiv y_2 \equiv y_3 \equiv 0 \pmod{\ell^t}$. We keep using the symbols $M, T$ from above, and set $H:=T-2v$. We shall prove that $t \leq t_{\max}-1$.

By definition, for every $w \in E_1[\ell^H]$ we have $\ell^H w = 0$, so for every $(g_1,g_2) \in H_\ell$ we have
\[
Mg_1w = Mg_1M^{-1} Mw = (g_2 M + O(\ell^H)) w=g_2Mw.
\]
It follows that the subgroup
$
\Gamma=\left\{ (w,Mw) \bigm| w \in E_1[\ell^H] \right\}
$
of $E_1 \times E_2$ is defined over $K_\ell$: indeed for any $(g_1,g_2) \in H_\ell$ and $(w,Mw) \in \Gamma$ we have
\[
(g_1,g_2) \cdot (w,Mw)=(g_1w,g_2Mw)=(g_1w,Mg_1w) \in \Gamma.
\]

Thus the abelian variety $A^*=\left(E_1 \times E_2\right)/\Gamma$ is defined over $K_\ell$, and the canonical projection $E_1 \times E_2 \to A^*$ is an isogeny of degree $|E_1[\ell^H]|=\ell^{2H}$; on the other hand, since we are in the situation of definition \ref{def_b0} we also have an isogeny $A^* \to E_1 \times E_2$ of degree $b$ dividing $b_0(E_1 \times E_2/K_\ell)$, and the composition of the two is an endomorphism $\varphi$ of $E_1 \times E_2$ that kills $\Gamma$. Since $E_1, E_2$ are not isogenous over $\overline{K}$, such an endomorphism corresponds to a pair $(\varphi_1,\varphi_2)$ where $\varphi_i$ is an endomorphism of $E_i$, and the assumption $\operatorname{End}_{\overline{K}}(E_i)=\mathbb{Z}$ implies that each $\varphi_i$ is multiplication by an integer. Moreover, since at least one of the coefficients of $M$ is an $\ell$-adic unit, we deduce that the projection of $\Gamma$ on $E_2$ contains at least one point of exact order $\ell^H$, so $\varphi$, which kills $\Gamma$, must be of the form $\left(\begin{matrix} \ell^H e_1 & 0 \\ 0 & \ell^H e_2 \end{matrix} \right)$ for some $e_1, e_2 \in \mathbb{Z}$. 
It follows that ${e_1}^2{e_2}^2\ell^{4H}=\deg(\varphi)=\ell^{2H}b$, hence we have $2H \leq v_\ell(b) \leq v_\ell(b_0(E_1 \times E_2/K_\ell))$ and $2t \leq v_\ell(b_0(E_1 \times E_2/K_\ell)) + 2(11n_1+n_2+7v+6)$. This inequality is not satisfied for any $t \geq t_{\max}$. 
\end{proof}

Combined with corollary \ref{cor_EasyCase}, the last proposition shows that $\ElThree$ contains $\ell^{f_1}\mathfrak{sl}_2(\mathbb{Z}_\ell)$, where 
\[
 f_1=\left\lfloor \frac{v_\ell(b_0(E_1 \times E_2/K_\ell))}{2} \right\rfloor+ 11n_1+5n_2+11v+7,\] 
and therefore $L(H_\ell)$ contains $0 \oplus \ell^{f_1} \mathfrak{sl}_2(\mathbb{Z}_\ell)$. Exchanging the roles of $E_1$ and $E_2$ and repeating the whole argument, we deduce that $L(H)$ contains $\ell^{f}\mathfrak{sl}_2(\mathbb{Z}_\ell) \oplus \ell^{f}\mathfrak{sl}_2(\mathbb{Z}_\ell)$, where now
\begin{equation}\label{eq_f}
f=\left\lfloor \frac{v_\ell(b_0(E_1 \times E_2/K_\ell))}{2} \right\rfloor+ 16\max\left\{n_1,n_2\right\} +11v + 7.
\end{equation}
We now have all we need to prove the following proposition:

\begin{proposition}\label{prop_Finale2Curves}
Let $E_1,E_2$ be elliptic curves over $K$ that are not isogenous over $\overline{K}$ and do not admit complex multiplication over $\overline{K}$. Let $\ell$ be a prime number.

Suppose the image of $\abGal{K_\ell} \to \operatorname{Aut}(T_\ell(E_i))$ contains $\mathcal{B}_\ell(n_i)$ for $i=1,2$ (where $n_i \geq 2$ for $\ell=2$ and $n_i \geq 1$ for $\ell=3$ or $5$). Let $f$ be given by formula \eqref{eq_f}. 
If $\ell$ is odd, the image $G_\ell$ of $\abGal{K} \to \operatorname{Aut}(T_\ell(E_1)) \times \operatorname{Aut}(T_\ell(E_2))$ contains $\mathcal{B}_\ell(4f, 4f)$; if $\ell=2$, the image $G_2$ of $\abGal{K} \to \operatorname{Aut}(T_2(E_1)) \times \operatorname{Aut} (T_2(E_2))$ contains 
\[\mathcal{B}_2(12(f+17\max\{n_1,n_2\}+13)+1,12(f+17\max\{n_1,n_2\}+13)+1).\]
\end{proposition}
\begin{proof}
For $\ell=2$ the result follows at once from theorem \ref{thm_PinkGL222}. For odd $\ell$, and under assumption \ref{assumpt_GaloisIsSmall}, the result similarly follows from property $(\ast)$ of $H_\ell$ given in theorem \ref{thm_PinkGL22} and the fact that clearly $2f > 8 \max\left\{n_1,n_2\right\}$. On the other hand, if assumption \ref{assumpt_GaloisIsSmall} does \textit{not} hold, then $G_\ell$ contains $\mathcal{B}_\ell\left( 20\max\{n_1,n_2\},20\max\{n_1,n_2\} \right)$ (note that we can assume $m_1 \leq n_1, m_2 \leq n_2$ without loss of generality, cf.~remark \ref{rmk_ValuesofMandN}), which is stronger than what we are claiming.
\end{proof}

\section{Conclusion}\label{sect_Conclusion}
Consider again the case of two elliptic curves $E_1,E_2$ defined over $K$, non-isogenous over $\overline{K}$ and such that $\operatorname{End}_{\overline{K}}(E_i)=\mathbb{Z}$. Let $\mathcal{P}$ be the set of primes $\ell$ for which $G_\ell$ does not contain $\operatorname{SL}_2(\mathbb{Z}_\ell)^2$.
\begin{lemma}
Let $\ell$ be a prime. If $\ell$ does not divide the product 
\[
30b_0(E_1/K;60) b_0\left(E_1^2/K;2\right) b_0(E_2/K;60) b_0\left(E_2^2/K;2\right) b_0(E_1 \times E_2/K;2),
\]
then $\ell$ is not in $\mathcal{P}$.
\end{lemma}

\begin{proof}
\cite[Lemma 8.2]{AdelicEC} implies that for a prime $\ell$ that does not divide
\[
b_0(E_1/K;60) b_0\left(E_1^2/K;2\right) b_0(E_2/K;60) b_0\left(E_2^2/K;2\right)
\]
both projections of $G_\ell(\ell)$ on the two factors $\operatorname{GL}_2(\mathbb{F}_\ell)$ contain $\operatorname{SL}_2(\mathbb{F}_\ell)$. Under this hypothesis, the proof of \cite[Proposition 1]{MR1209248} shows that $G_\ell(\ell)$ contains $\operatorname{SL}_2(\mathbb{F}_\ell)^2$ unless $\ell^2 \bigm| b_0(E_1 \times E_2/K;2)$. Finally, when $\ell \geq 5$ (as is the case here, thanks to the factor 30 appearing in the product) a closed subgroup of $\operatorname{GL}_2(\mathbb{Z}_\ell)^2$ whose projection modulo $\ell$ contains $\operatorname{SL}_2(\mathbb{F}_\ell)^2$ contains all of $\operatorname{SL}_2(\mathbb{Z}_\ell)^2$ (this is well-known; see for example \cite[Proposition 4.2]{MR1610883}).
\end{proof}


\begin{corollary}\label{cor_ProdBadPrimes}
We have
\[
\prod_{\ell \in \mathcal{P}} \ell \leq 30 b_0(E_1/K;60) b_0\left(E_1^2/K;2\right) b_0(E_2/K;60) b_0\left(E_2^2/K;2\right) b_0(E_1 \times E_2/K;2).
\]
\end{corollary}
Let now $\ell$ be a prime different from $2$ and $3$, and for $j=1,2$ set
\[
D_j(\infty)=b_0(E_j/K;24)^5 b_0(E_j^2/K;24).
\]
Since clearly $\ell^{v_\ell(D_j(\infty))+1}$ does not divide $D_j(\infty)$, we see from \cite[Theorem 7.5]{AdelicEC} that $G_{\ell,j}$ contains 
$\mathcal{B}_\ell\left(16 v_\ell(D_j(\infty))+12 \right),$
 hence the same is true for $H_{\ell,j}$, cf.~remark \ref{rmk_ValuesofMandN}. Therefore -- in the notation of the previous section -- we can take $n_j=n_j(\ell)=16 v_\ell(D_j(\infty))+12$ (this obviously satisfies the condition $n_1, n_2 \geq 1$ for $\ell=5$ -- cf.~definition \ref{def_n1n2}). For $\ell=3$, using the fact that our group $H_{3,j}$ is the group $H_3$ of \cite{AdelicEC} we see (again by \cite[Theorem 7.5]{AdelicEC}) that we can take
 \[
 n_j(3) = 16  v_3\left(b_0(E/K_3)^5 b_0(E^2/K_3)\right)+12 \leq 16 v_3(D_j(\infty))+12;
 \]
similarly, for $\ell=2$ we can take $n_j(2) = 48v_2\left(b_0(E_j/K_2)^5 b_0(E_j^2/K_2)\right)+38.$
Applying proposition \ref{prop_Finale2Curves} with these values of $n_j(\ell)$ we get:

\begin{lemma}\label{lemma_Estimatesb0}
Let $\ell$ be a prime. The group $G_\ell$ contains $\mathcal{B}_\ell(f(\ell),f(\ell))$, where $f(\ell)$ is given by
\[
f(\ell)=2v_\ell(b_0(E_1 \times E_2/K;2\cdot 48^2))+2^{10} \max\left\{v_\ell(D_1(\infty)),v_\ell(D_2(\infty)) \right\}+800
\]
for odd $\ell$ and
\[
f(2) = 6 v_2(b_0(E_1 \times E_2/K_2))+19008 \max_j \left\{v_2\left(b_0(E_j/K_2)^5 b_0(E_j^2/K_2)\right) \right\}+ 15421
\]
for $\ell=2$.
\end{lemma} 



Using the very same argument as in \cite[§9]{AdelicEC}, and some very crude estimates, we deduce
\begin{proposition}\label{prop_ProductDecomposition}
Denote by $G_\infty$ the image of $\abGal{K}$ inside
\[\prod_{\ell} \left( \operatorname{Aut} T_\ell(E_1) \times \operatorname{Aut} T_\ell(E_2) \right) \subset \operatorname{GL}_2(\hat{\mathbb{Z}})^2.\]

$G_\infty$ contains a subgroup $S$ of the form $S=\prod_{\ell} S_\ell$, where each $S_\ell$ coincides with $\operatorname{SL}_2(\mathbb{Z}_\ell)^2$ except for the finitely many primes that are in $\mathcal{P}$, for which $S_\ell=\mathcal{B}_\ell(f(\ell),f(\ell))$. The index of $S$ in $\operatorname{SL}_2(\hat{\mathbb{Z}})$ is bounded by $\displaystyle b(E_1 \times E_2/K;2 \cdot 48^2)^{10^4}$.
\end{proposition}

We finally come to the adelic estimate for an arbitrary number of curves:
\begin{theorem}\label{thm_FinaleNCurves}
Let $E_1, \ldots, E_n$, $n \geq 2$, be elliptic curves defined over $K$, pairwise non-isogenous over $\overline{K}$. Suppose that $\operatorname{End}_{\overline{K}}(E_i)=\mathbb{Z}$ for $i=1,\ldots,n$ and let $G_\infty$ be the image of the natural representation
\[
\rho_\infty : \abGal{K} \to \prod_{i=1}^n \prod_{\ell} \operatorname{Aut} T_\ell(E_i) \cong \operatorname{GL}_2(\hat{\mathbb{Z}})^n.
\]
Then $G_\infty$ has index at most
\[
8^{n(n-2)} \zeta(2)^{n(n-1)} \cdot [K:\mathbb{Q}] \cdot \max_{i \neq j} b\left(E_i \times E_j/K;2 \cdot 48^2\right)^{5000 n(n-1)}
\]
in
\[
\Delta=\left\{ (x_1,\ldots,x_n) \in  \operatorname{GL}_2(\hat{\mathbb{Z}})^n \bigm\vert \det x_i = \det x_j \quad \forall i,j \right\}.
\]
\end{theorem}

\begin{proof}
The Galois-equivariance of the Weil pairing implies that for every index $i=1,\ldots,n$ an element $g \in \abGal{K}$ acts on $T_\ell (E_i)$ through an automorphism of determinant $\chi_\ell(g)$. This immediately implies that $G_\infty \subseteq \Delta$. With a slight abuse of language, we write $\det$ for the map $\Delta \to \hat{\mathbb{Z}}$ sending $(x_1,\ldots,x_n)$ to $\det x_1$. 
The exact sequence
\[
1 \to G_\infty \cap \operatorname{SL}_2(\hat{\mathbb{Z}})^n \to G_\infty \stackrel{\det}{\longrightarrow} \hat{\mathbb{Z}}^\times \to \frac{\hat{\mathbb{Z}}^\times}{\det (G_\infty)}  \to 1
\]
and the inequality $\left| \displaystyle \frac{\hat{\mathbb{Z}}^\times}{\det (G_\infty)} \right| \leq [K:\mathbb{Q}]$ (cf.~proposition \cite[Proposition 8.1]{AdelicEC}) show that in order to establish the theorem it is enough to prove that the index of $G_\infty \cap \operatorname{SL}_2(\hat{\mathbb{Z}})^n$ inside $\operatorname{SL}_2(\hat{\mathbb{Z}})^n$ is bounded by
\[
8^{n(n-2)}\zeta(2)^{n(n-1)} \cdot \max_{i \neq j} b(E_i \times E_j/K;2 \cdot 48^2)^{5000 n(n-1)}.
\]

Set $G=G_\infty \cap \operatorname{SL}_2(\hat{\mathbb{Z}})$. For every pair $E_i, E_j$ of curves, we get from proposition \ref{prop_ProductDecomposition} a subgroup $S^{(i,j)}$ of
\[
\operatorname{SL}_2(\hat{\mathbb{Z}})^2 \subseteq \prod_{\ell} \operatorname{Aut} \left(  T_\ell(E_i) \right) \times \prod_{\ell} \operatorname{Aut} \left(  T_\ell(E_j) \right)
\]
that satisfies all the requirements of corollary \ref{cor_ProductsOfManyCurves}, and the theorem follows from this corollary upon recalling that the index of $S^{(i,j)}$ in $\operatorname{SL}_2(\hat{\mathbb{Z}})^2$ is bounded by 
$b(E_i \times E_j/K;2 \cdot 48^2)^{10000}$ .
\end{proof}

\bibliography{Biblio}{}
\bibliographystyle{plain}

\end{document}